\documentclass{siamart251216}
\usepackage{verbatim}
\usepackage{mathrsfs}
\usepackage{amsmath}
\usepackage{amssymb}
\usepackage{graphicx}

\makeatletter
\usepackage{mathrsfs}
\usepackage{float}

\newtheorem{thm}{Theorem}
\newtheorem{rem}{Remark}
\newtheorem{cor}{Corollary}
\newtheorem{prop}{Proposition}
\newtheorem{example}{Example}
\newtheorem{lem}{Lemma}
\newtheorem{defn}{Definition}

\usepackage{babel}

\makeatother

\begin{document}
\title{Stable Dynamical Mass Transportation via port-Hamiltonian Structure}
\author{Dongjun Wu 
	\thanks{D. Wu is with the Department of Electrical Engineering, ESAT-STADIUS, KU Leuven, Leuven, Belgium}
	\and 
	Yankai Lin
	\thanks{Y. Lin is with the School of Computer Science and Engineering, Wuhan Institute of Technology,
	Wuhan, China}
	\and
	Gang Feng
	\thanks{G. Feng is with the Department of Mechanical Engineering, City University of Hong Kong, Hong
	Kong, China}
}

\maketitle
\begin{abstract}
Dynamical optimal mass transportation over nonlinear systems is often hindered by the curse of dimensionality and open-loop implementation. We propose an alternative framework that sacrifices exact optimality in favor of stability, robustness, and scalability. We use a stochastic port-Hamiltonian system as a reference model. Under suitable structural assumptions, we show that its density converges asymptotically to a prescribed steady-state distribution. We then shape a deterministic nonlinear system into a compatible port-Hamiltonian form and add a local density-feedback term so that its closed-loop continuity equation matches the Fokker-Planck-Kolmogorov equation of the reference model. The resulting feedback drives a broad class of initial distributions to the target distribution. Because the density term is local at the particle level, the controller admits decentralized implementation and avoids solving a global finite-horizon transport problem online. Nonlinear examples and simulations illustrate the method. For linear systems with Gaussian targets, we derive necessary and sufficient conditions for stabilizing density feedback and give explicit controller constructions.
\end{abstract}

\section{Introduction }\label{sec:intro}

Controlling the evolution of probability densities has emerged as
an important problem in modern control theory, which enables the regulation
of uncertainty and collective behavior in complex systems ranging
from swarm of robotics \cite{elamvazhuthi2019mean,Maffettone2023}
to traffic flows \cite{lighthill1955kinematic,richards1956shock}
and mean field games \cite{lasry2007mean,bensoussan2013mean}. By
shifting focus from pointwise state regulation (known as the Lagrangian
viewpoint) to shaping full probability distributions (known as the Eulerian
viewpoint) over time, density control frameworks provide new insights
and tools for control and optimization of large scale dynamical systems. 

In parallel with density control, there has been a rapidly growing
interest in dynamical optimal mass transportation (DOT) within the
control community \cite{chen2021optimal,elamvazhuthi2023dynamical,bunne2023optimal,agrachev2009optimal,wu2025duality}.
Building upon the Benamou-Brenier fluid dynamic formulation of Monge-Kantorovich
mass transportation \cite{benamou2000computational} and the theory
of Schrödinger bridges problems, recent work has recast a variety
of distribution-steering problems as DOT formulations \cite{chen2021stochastic}.
Beyond this, it has been shown that many classical topics in systems
and control -- including linear quadratic Gaussian control, covariance
steering, and ensemble control -- admit natural interpretations within
the DOT framework, see the survey paper \cite{chen2021stochastic}
and the references therein. As a result, DOT has become a powerful
control-theoretic tool: it underpins algorithms for steering state
distributions over linear \cite{chen2016optimal,chen2015optimal_I,chen2015optimal_II,Chen2018_optimal_III}
and nonlinear dynamics \cite{elamvazhuthi2018optimal,wu2025nonlinear,agrachev2009optimal},
and provides a unifying language for density control, mean field control,
and related distributional problems in systems and control.

Despite these advantages, DOT also comes with important limitations
from a control-theoretic perspective. First, most formulations are
inherently finite-horizon: they focus on steering an initial distribution
to a terminal one over a finite time interval \cite{benamou2000computational,ghoussoub2018optimal,chen2021optimal}.
This structure is not well suited to regulation or stabilization tasks,
where the objective is to maintain a desired distribution over an
indefinite horizon rather than achieve a one-shot transfer. Second,
standard DOT formulations assume perfect knowledge of the initial
distribution and typically yield open-loop control inputs \cite{chen2016optimal,Chen2018_optimal_III},
which is sensitive to model uncertainty
and external disturbance. These issues are
particularly restrictive in applications where the objective is not
exact finite-time optimality, but rather stable, robust, and scalable
regulation of a population distribution.

This paper proposes an alternative framework for dynamical mass transportation
that sacrifices exact transportation optimality in order
to gain stability, robustness, and scalability. The main idea is to  combine stochastic port-Hamiltonian
reference dynamics with density-feedback control of deterministic systems. We first identify a class of
stochastic port-Hamiltonian systems whose structural properties guarantee convergence toward a prescribed
steady-state distribution. We then design a density-feedback law for a deterministic control system so that
its closed-loop continuity equation coincides with the Fokker-Planck-Kolmogorov equation of the chosen
stocha\-stic reference system. As a result, the density of the deterministic closed-loop system converges asymptotically to the desired distribution.

The proposed framework transforms the mass transportation problem to
two more tractable subproblems. The first is a structure-matching subproblem:
to seek a feedback law that shapes the deterministic dynamics into
a port-Hamiltonian form with adequate dissipation and interconnection
structure. This subproblem falls natur\-ally within the framework of interconnection
and damping assignment passivity-based control (IDA-PBC) \cite{Ortega2002},
which is known for its physical interpretation \cite{VanderSchaft1999}
and robustness \cite{Ortega2002}. The second subproblem is a density-feedback
term that emulates diffusion at the level of the density evolution.
Notably, this feedback term is local in the microscopic state
variable and can be approximated using nearby particles. Therefore,
the corresponding controller is particle-wise and decentralized, avoiding
the need to solve a global finite-horizon optimal transportation problem
online.

In fact, the proposed approach is similar in spirit to diffusion-based
generative mechanisms: although the induced probability flow is not
globally optimal in the transport sense, it provides a stable and
robust mechanism for shaping distributions \cite{lavenant2022flow}.
The emphasis of this work is thus different from that of classical DOT. Rather
than computing an optimal transfer between two prescribed distributions,
we design a feedback mechanism that stabilizes a desired distribution
and corrects deviations dynamically. This feedback viewpoint is essential
for robustness, while the locality of the density feedback provides
a route toward scalable implementation.

Similar to our works are Fokker-Planck-Kolmogorov control \cite{Annunziato2018,Breiten2018}
and PDE-based density control \cite{Elamvazhuthi2018,Sinigaglia2025}.
These approaches typically optimize, track, or stabilize probability
densities governed by FPK, continuity, or mean-field equations. They
are therefore closely related to the present work at the level of
density dynamics. The differences lie in two aspects: 1) most existing
formulations either treat the density equation itself as the controlled
object or solve a PDE-constrained optimization probl\-em. They have not addressed the problem 
of  how to design density controllers directly from
the structure of the underlying nonlinear control system, while retaining
stability, robustness, and scalable particle-wise implementation.
2) These works  consider stochastic control or density control
and do not use a stochastic system as a reference model to solve
mass transportation problems.

The main contributions of this paper are twofold.
\begin{enumerate}
\item We propose a new framework for mass transportation of nonlinear systems
based on stochastic port-Hamiltonian structure, IDA-PBC, and density
feed\-back. The framework provides stable and robust density regulation,
and the resulting feedback law can be implemented in a decentralized
particle-wise manner. Several nonlinear examples are presented to
illustrate how the proposed method can be applied beyond the linear-Gaussian
setting.
\item We establish a universal result for linear systems with Gaussian
target distrib\-utions. In particular, we show that a linear system
can be transported from a non-Gaussian distribution to a non-degenerate
Gaussian distribution in a stable way if and only if it can be shaped
into a port-Hamiltonian structure satisfying the required structural
properties. This yields necessary and sufficient conditions for stabilizing
mass transportation schemes and leads to explicit controller constructions.
\end{enumerate}

\par

The paper is organized as follows. In Section \ref{sec:Problem-formulation},
we introduce the problem formulation. In Section \ref{sec:sto-pH},
we study a class of stochastic port-Hamiltonian systems and prove
that such systems converge asymptotically to prescribed steady-state distributi\-ons.
In Section \ref{sec:OT-determ}, we use these results to design controllers
for mass transportation of nonlinear systems. Section \ref{sec:lin-sys}
specializes the framework to linear systems with Gaussian target distributions,
where sharper characterizations can be obtained. In Section \ref{sec:NL-exmp},
we present several nonlinear examples to demonstrate the use of the
proposed method. Section \ref{sec:simu} provides numerical simulations
illustrating the effectiven\-ess of the proposed framework.

\emph{Notations}: We use $\mathscr{P}(\mathbb{R}^{n})$ to denote
probability distributions on $\mathbb{R}^{n}$, and $\mathscr{P}_{+}(\mathbb{R}^{n})$
the set of strictly positive probability densities. Denote $\mathscr{N}(m,\Sigma)$
as the normal distribution of mean $m$ and covariance $\Sigma$.
The operator $\nabla$ (or $\nabla_{x}$ with $x$ being the variable)
is the standard gradient operator, and ${\rm div}$ the diverngence
operator. For a matrix $A$, ${\rm Tr}(A)$ is the trace of $A$,
${\rm Im}A$ stands for the column space of $A$, and $A^{\dagger}$ is
the pseudo inverse of $A$. We write $A\succeq0$ ($A\succ0$) if
$A$ is symmetric positive semi-definite (definite). Given two smooth
vector fields $f$ and $g$, denote $[f,g]$ the Lie bracket between
the two vector fields.

\section{Problem formulation }\label{sec:Problem-formulation}

We consider a nonlinear system
\begin{equation}
\dot{x}=f(x)+g(x)u, \label{sys:NL}
\end{equation}
where $x\in\mathbb{R}^{n}$ is the state and $u\in\mathbb{R}^{m}$
is the control input. We suppose the system describes the microscopic
dynamics of an ensemble of agents, each evolving according to the
same dynamics. The aggregate behavior of the ensemble is described
by a probability density $\rho(t,x)$, which represents the distribution
of the population states at time $t$. Under a control law $u(t,x)$,
the density evolves according to the continuity equation
\begin{equation}
\partial_{t}\rho+{\rm div}(\rho(f+gu))=0, \label{sys:CE}
\end{equation}

The central problem studied in this paper is the following: 
\begin{quote}
\emph{Given a desired probability density $\rho_{d}$, design a control
law $u(t,x)$}\footnote{It would have been more accurate to write the controller as $u(t,x,\rho)$
since we will use explicitly the density as a feedback. However, the
controller only uses local density information, i.e., $u(t,x,\rho)=u(t,x,\rho(t,x))$,
we choose to write as $u(t,x)$ for simplicity. }\emph{ such that the solution $\rho(t,x)$ of the system \eqref{sys:CE}
converges to $\rho_{d}$ as $t\to\infty$ for every admissible initial
density in a suitable topology.}
\end{quote}

Here, an admissible initial density means a probability density for
which the closed-loop density dynamics is well posed and the KL-divergence
$\mathrm{KL}(\rho_{0}\|\rho_{d})$ is finite. The precise definition
of the KL-divergence will be given in the next section. 

This problem is different from finite-horizon density steering or
dynamical optimal transport. In finite-horizon formulations, one typically
seeks a control law or transport plan that transfers a prescribed
initial density $\rho_{0}$ to a prescribed terminal density $\rho_{d}$
at a fixed final time. Such formulations are appropriate for one-shot
transportation tasks, but they do not directly provide a stabilizing
feedback mechanism. In particular, the corresponding control depends on
the chosen initial distribution and terminal time, and additional
design is needed to maintain the target density after the terminal
time or to correct deviations caused by perturbations, disturbances,
or modeling errors.

The objective of this paper is therefore not to construct an optimal
finite-time transfer between two prescribed densities. Instead, we
seek a feedback mechanism for which $\rho_{d}$ becomes an asymptotically
stable steady-state density of the closed-loop density dynamics. The
desired controller should dynamically correct deviations of the current
density from $\rho_{d}$, and should be implementable at the particle
level whenever the density or its local approximation is available.

\section{Stochastic port-Hamiltonian systems}\label{sec:sto-pH}
\subsection{Stochastic port-Hamiltonian reference dynamics}

The purpose of this section is to introduce the stochastic port-Hamiltonian
dynamics that will serve as a reference model for stable mass transportation.
In the deterministic setting, port-Hamiltonian systems combine an
energy function, an interconnection structure, and a dissipation mechanism.
Their stability properties are therefore encoded in the energy-dissipation
structure. Here we use the same idea at the level of probability densities:
the stochastic port-Hamiltonian dynamics will be chosen so that it
admits a prescribed Gibbs-type steady-state distribution and dissipates
the relative entropy (or the KL-divengence) with respect to this distribution.

Consider the following stochastic port-Hamiltonian system introduced
in \cite{Satoh2013}:
\begin{equation}
dx=(J(x)-R(x))\nabla H(x)dt+h(x)dW, \label{sys:PH}
\end{equation}
where $W$ is a standard Brownian motion in $\mathbb{R}^{d}$, $J(x)=-J(x)^{\top}$
is the so-called interconnection matrix, $R(x)=R(x)^{\top}\succeq0$
a damping matrix, $H(x)$ a scalar-valued function called the storage
function, and $h(x)=[h_{1}(x),\cdots,h_{d}(x)]\in\mathbb{R}^{n\times d}$
is a diffusion term. The probability density of the system \eqref{sys:PH}
evolves according to the Fokker-Planck-Kolmogorov (FPK) equation,
i.e., let $\rho(t,\cdot)$ be the density function of the ensemble at time $t$, then
\begin{equation}
\partial_{t}\rho=-{\rm div}\left\{ \rho(J-R)\nabla H-\frac{1}{2}hh^{\top}\nabla\rho\right\} .\label{eq:FPK}
\end{equation}

The Gibbs-type probability density associated with the system \eqref{sys:PH}
is
\begin{equation}
\pi(x)=\frac{e^{-H(x)}}{\int_{\mathbb{R}^{n}}e^{-H(x)}{\rm d}x}.\label{eq:pi}
\end{equation}
Our objective in this section is to identify structural conditions
under which $\pi$ is not only invariant, but also asymptotically
stable under adequate structural assumptions.

Throughout this section, we assume that the functions $J$, $R$,
$H$ and $h$ are sufficien\-tly smooth and that \eqref{eq:FPK} generates
a well-defined density flow. We also impose the following important
assumptions.

\textbf{Assumption H1. }The function $H(\cdot)$ satisfies
\[
\int_{\mathbb{R}^{n}}e^{-H(x)}{\rm d}x<\infty.
\]

\textbf{Assumption H2}. The interconnection matrix $J$ is divergence
free in the sense that for every $j=1,\cdots,n$, 
\[
\sum^{n}_{i=1}\frac{\partial J_{ij}}{\partial x_{i}}=0,
\]
and the damping matrix satisfies
\[
R(x)=\frac{1}{2}h(x)h(x)^{\top}.
\]

\textbf{Assumption H3.} The density flow generated by \eqref{eq:FPK}
preserves strict positivity: for every admissible initial density
$\rho_{0}$ that is strictly positive on $\mathbb{R}^{n}$, $\rho(t,\cdot)$
is strictly positive for all $t>0$. 

\textbf{Assumption H4.} Along the FPK equation \eqref{eq:FPK}, the
largest invariant subset of 
\[
\mathscr{I}=\left\{ \rho\in\mathscr{P}_{+}(\mathbb{R}^{n}):h(\cdot)^{\top}\nabla\ln\frac{\rho(\cdot)}{\pi(\cdot)}=0\right\} 
\]
is the singleton $\{\pi\}$. 

We briefly comment on the seemingly non-trivial Assumptions H2 - H4:
\begin{itemize}
\item For Assumption H2, the simplest case is to set $J$ as a constant
skew-symmetric matrix. For $n=3$, one can construct $J$ from potential
functions: Let $\phi$ be a smooth real valued function. Consider
\begin{equation}
J(x)=\begin{bmatrix}0 & -\partial_{3}\phi & \partial_{2}\phi\\
\partial_{3}\phi & 0 & -\partial_{1}\phi\\
-\partial_{2}\phi & \partial_{1}\phi & 0
\end{bmatrix}.\label{eq:div-free J}
\end{equation}
Then $J$ is divergence free. For $n>3$, one can construct $J$ from
$3$-tensors $A=(A_{ijk})$ such that $A_{kij}=-A_{kji}$ and let
$J_{ij}=\sum^{n}_{k=1}\partial_{k}A_{kij}$.
\item Assumption H3 is a ``non-degeneracy'' condition, which can be guaranteed
by the so called Hörmander condition \cite{hormander1967}, or equivalently
the so called hypoelliptic condition \cite{villani2009hypocoercivity}.
To state the condition, let us first turn the system \eqref{sys:PH}
into Stratonovich form,
\[
dx=V_{0}(x)dt+\sum^{d}_{i=1}V_{i}(x)\circ dW_{i},
\]
where $\{ W_i \}_{i=1}^d $ are independent scalar standard Brownian motions, $V_{i}(x)=h_{i}(x)$ for $i=1,\cdots,d$, and 
\begin{align*}
V_{0} & =(J-R)\nabla H-\frac{1}{2}\sum^{d}_{i=1}\frac{\partial h_{i}}{\partial x}h_{i}.
\end{align*}
Consider the Lie algebra generated by $\{V_{i}\}^{d}_{0}$:
\[
V_{i},[V_{i},V_{j}],[V_{i},[V_{j},V_{k}]],\cdots
\]
which is denoted as ${\rm Lie}(V_{0},\cdots,V_{d})$, where $[\cdot,\cdot]$
denotes the Lie bracket betwe\-en vector fields. Then the system \eqref{sys:PH}
is said to satisfy the Hörmander condition at $x$ if
\[
{\rm span}\{Y(x):Y\in{\rm Lie}(V_{0},\cdots,V_{d})\}=\mathbb{R}^{n}.
\]
In particular, the condition holds for linear systems when controllability-type condition is imposed.

\item Assumption H4 is the most difficult one to verify and it has to be
verified case by case. This is similar to verifying asymptotic
stability of port-Hamiltonian systems in the IDA-PBC framework.
\end{itemize}
To measure the difference between probability densities, we use Kullback-Leibler
(KL) divergence. Let $\rho_{1}$ and $\rho_{2}$ be two probability
densities on $\mathbb{R}^{n}$ such that $\rho_{1}$ is aboslutely
continuous with respect to $\rho_{2}$. The \emph{KL-divergence} of
$\rho_{1}$ with respect to $\rho_{2}$ is defined as 
\[
{\rm KL}(\rho_{1}\|\rho_{2}):=\int_{\mathbb{R}^{n}}\ln\left(\frac{\rho_{1}(x)}{\rho_{2}(x)}\right)\rho_{1}(x){\rm d}x.
\]
It is well-known that ${\rm KL}(\rho_{1}\|\rho_{2})\ge0$ and that
${\rm KL}(\rho_{1}\|\rho_{2})=0$ if and only if $\rho_{1}(x)=\rho_{2}(x)$
almost everywhere. 

In order to state convergence, we equip $\mathscr{P}(\mathbb{R}^{n})$
with the \emph{narrow topology}. Since we work with probability densities,
we identify a density $\rho$ with the probability measure $\rho(x){\rm d}x$.
A sequence of probability densities $\{\rho_{k}\}_{k\geq1}$ is said
to converge narrowly to $\bar{\rho}$ if, for every bounded continuous
function $\varphi:\mathbb{R}^{n}\to\mathbb{R}$, there holds
\[
\lim_{k\to\infty}\int_{\mathbb{R}^{n}}\varphi(x)\rho_{k}(x){\rm d}x=\int_{\mathbb{R}^{n}}\varphi(x)\bar{\rho}(x){\rm d}x.
\]
We write $\rho_{k}\rightharpoonup\bar{\rho}$ when $\rho_{k}$ converges
to $\bar{\rho}$ in the narrow topology.

We can now state the main convergence result for the stochastic port-Hamiltonian
reference dynamics:
\begin{thm}
\label{thm:main}Suppose that the system \eqref{sys:PH} satisfies
Assumptions H1 - H4, then for any admissible initial $x_{0}\sim\rho_{0}$
with ${\rm KL}(\rho_{0}\|\pi)<\infty$ where $\pi$ is defined as
\eqref{eq:pi}, the probability density of the system \eqref{sys:PH}
converges to $\pi$ asymptotically in the narrow topology. 
\end{thm}
The next subsection is devoted to the proof of this theorem.

\subsection{Proof of Theorem \ref{thm:main}}

Consider the candidate Lyapunov function $\mathcal{F}(\rho)={\rm KL}(\rho\|\pi)$.
We divide the proof into several steps.

\emph{Step 1}: Assume ${\rm KL}(\rho_{0}\|\pi)<\infty$. Denote for
convenience $\rho_{t}:=\rho(t,\cdot)$. We show that 
\[
\mathcal{F}(\rho_{t}))\le\mathcal{F}(\rho_{0}),
\]
for all $t\ge0$. 

Compute the time derivative of $\mathcal{F}$ along the FPK equation
\eqref{eq:FPK}:{\small
\begin{align}
 & \frac{d\mathcal{F}(\rho)}{dt}\nonumber \\
= & \int_{\mathbb{R}^{n}}(\partial_{t}\rho)\ln\frac{\rho}{\pi}+\rho\frac{\partial_{t}\rho/\pi}{\rho/\pi}{\rm d}x\\
= & \int_{\mathbb{R}^{n}}(\partial_{t}\rho)\left(\ln\frac{\rho}{\pi}+1\right){\rm d}x\nonumber \\
= & \int_{\mathbb{R}^{n}}\left(-\nabla\cdot\rho(J-R)\nabla H+\frac{1}{2}\nabla\cdot hh^{\top}\nabla\rho\right)\left(\ln\frac{\rho}{\pi}+1\right){\rm d}x\nonumber \\
= & \int_{\mathbb{R}^{n}}[\rho(J-R)\nabla H-R\nabla\rho)]\cdot\nabla\ln\frac{\rho}{\pi}{\rm d}x\nonumber \\
= & \int_{\mathbb{R}^{n}}\rho J\nabla H\cdot\nabla\ln\frac{\rho}{\pi}{\rm d}x-\int_{\mathbb{R}^{n}}\rho R\nabla\ln\frac{\rho}{\pi}\cdot\nabla\ln\frac{\rho}{\pi}{\rm d}x\nonumber \\
= & -\frac{1}{2}\int_{\mathbb{R}^{n}}\rho\left\Vert h^{\top}\nabla\ln\frac{\rho}{\pi}\right\Vert ^{2}{\rm d}x, \label{eq:dF}
\end{align}
}where we have used the fact that
\begin{align*}
 & \int_{\mathbb{R}^{n}}\rho J\nabla H\cdot\nabla\ln\frac{\rho}{\pi}{\rm d}x\\
= & \int_{\mathbb{R}^{n}}J\nabla H\cdot\nabla\rho+\rho J\nabla H\cdot\nabla H{\rm d}x\\
= & \int_{\mathbb{R}^{n}}J\nabla H\cdot\nabla\rho{\rm d}x\text{ (skew-symmetry of }J)\\
= & -\int_{\mathbb{R}^{n}}\sum^{n}_{i,j=1}\partial_{i}(J_{ij})\partial_{j}H-\int_{\mathbb{R}^{n}}\rho{\rm Tr}(J\nabla^{2}H){\rm d}x\\
= & 0.
\end{align*}
The last equality follows from Assumption H4 and skew-symmetry of
$J$. From \eqref{eq:dF} we see that $\mathcal{F}(\rho_{t})$ is
non-increasing and hence $\mathcal{F}(\rho_{t})\le\mathcal{F}(\rho_{0})$
for all $t\ge0$.

\emph{Step 2}: We prove that the $\omega$-limit set of $\rho_{0}$
(see Definition \ref{def:System-Dyn} given in Appendix) along the FPK equation \eqref{eq:FPK}
is the singleton $\{\pi\}$. For any initial density $\rho_{{\rm ini}}$,
let $S_{t}\rho_{{\rm ini}}$ denote the solution to the FPK equation
\eqref{eq:FPK} at $t$. Let $D(\rho)=\int\rho\|h^{\top}\nabla\ln\frac{\rho}{\pi}\|{\rm d}x$.
Then $\{S_{t}\}_{t\ge0}$ forms a dynamical system (Definition \eqref{def:System-Dyn})
on $\mathscr{P}(\mathbb{R}^{n})$ equipped with the narrow topology. 

If $\bar{\rho}\in\omega(\rho_{0})$, then by definition of the $\omega$-limit
set, there exists a sequence $\{t_{n}\}_{n\ge1}$, $t_{n}\to\infty$
such that $\rho_{t_{n}}\rightharpoonup\bar{\rho}$ narrowly. 
Hence, $\rho_{t_{n}+s}\rightharpoonup S_{s}\bar{\rho}$ as $t_{n}\to\infty$
by the continuity of $S_{s}$ in the narrow topology. By lower semicontinuity
of $D(\cdot)$ \footnote{This is guaranteed by the fact that $h$ and $\pi$ are $C^{1}$ and
that $\pi>0$.}, 
\[
D(S_{s}\bar{\rho})\le\liminf_{n\to\infty}D(S_{s}\rho_{t_{n}})=\liminf_{n\to\infty}D(\rho_{t_{n}+s}).
\]
By Fatou's lemma \cite{folland1999}, for every $T>0$,
\begin{align*}
\int^{T}_{0}D(S_{s}\bar{\rho})ds & \le\liminf_{n\to\infty}\int^{T}_{0}D(\rho_{t_{n}+s}){\rm d}s\\
 & =\liminf_{n\to\infty}\int^{t_{n}+T}_{t_{n}}D(\rho_{s}){\rm d}s\\
 & =\liminf_{n\to\infty}\mathcal{F}(\rho_{t_{n}})-\mathcal{F}(\rho_{t_{n}+T})\\
 & =0,
\end{align*}
since $\mathcal{F}(\rho_{t})$ decreases monotonically. This implies
that $D(S_{s}\bar{\rho})=0$ for all $s\in[0,T]$. Since $T$ is arbitrary,
we conclude that $S_{s}\bar{\rho}\in\mathscr{I}$ (see Assumption
H4) for all $s\ge0$. By Assumption H2, we must have $\bar{\rho}=\pi$.

\emph{Step 3}: We show that the set $\cup_{t\ge0}\{\rho_{t}\}$ is
precompact in $\mathscr{P}(\mathbb{R}^{n})$.

It is well known that the level set of the entropy (the KL-divergence)
\[
Q=\{\rho\in\mathscr{P}(\mathbb{R}^{n}):\mathcal{F}(\rho)\le\mathcal{F}(\rho_{0})\}.
\]
is precompact in the narrow topology, see for example \cite{dembo2009large}.
Then since $\mathcal{F}(\rho_{t})\le\mathcal{F}(\rho_{0})$ for all
$t\ge0$, we have $\cup_{t\ge0}\{\rho_{t}\}\subset Q$ and thus the
claim follows.

\emph{Step 4}: In view of Step 2 and Step 3, the conclusion follows
by invoking the LaSalle invariance principle in infinite dimension,
i.e., Theorem \ref{thm:Lasalle} given in Appendix.\footnote{We point out that Theorem \ref{thm:Lasalle} assumes metric space
setting. Although the probability space $\mathscr{P}(\mathbb{R}^{n})$
is not a metric space, it is metrizable by the narrow topology, see
for example \cite{folland1999}. Thus the theorem can be applied to
our case.}

\section{Mass transportation of deterministic systems}\label{sec:OT-determ}

\subsection{Control by structure matching}

We now consider the mass transportati\-on problem for the deterministic
nonlinear system \eqref{sys:NL}. The main idea is to use the stochastic
port-Hamiltonian system \eqref{sys:PH} as a reference model. More
precisely, we seek a density-feedback controller such that the continuity
equation \eqref{sys:CE} coincides with the FPK equation \eqref{eq:FPK}
associated with \eqref{sys:PH}. Equivalently, we would like to impose
the flux-matching condition
\begin{equation}
\rho(f+gu)=\rho((J-R)\nabla H)+R\nabla\rho.\label{eq:match1}
\end{equation}

To achieve this, we decompose this matching equation into two parts
as in the following result. 
\begin{thm}
\label{thm:match}Suppose there exist smooth maps $v:\mathbb{R}^{n}\to\mathbb{R}^{m}$,
$J:\mathbb{R}^{n}\to\mathbb{R}^{n\times n}$, $R:\mathbb{R}^{n}\to\mathbb{R}^{n\times n}$
and $H:\mathbb{R}^{n}\to\mathbb{R}$ such that 
\begin{equation}
f(x)+g(x)v(x)=(J(x)-R(x))\nabla H(x), \label{eq:match2}
\end{equation}
with $J(x)=-J(x)^{\top}$, $R(x)=R(x)^{\top}\succeq0$ and 
\begin{equation}
{\rm Im}R(x)\subset{\rm Im}g(x), \label{eq:match3}
\end{equation}
for every $x\in\mathbb{R}^{n}$. Suppose moreover, the parameters
$J,R,H$ are such that Assumption H1 - H4 are satisfied over the FPK
equation \eqref{eq:FPK}. Then the following controller
\begin{equation}
u(t,x)=v(x)-g(x)^{\dagger}R(x)\nabla\ln\rho(t,x)\label{eq:closedu}
\end{equation}
steers every admissible positive initial density $\rho_{0}$ with
${\rm KL}(\rho_{0}\|\pi)<\infty$ to $\pi$ (see \eqref{eq:pi}) asymptotically
in the narrow topology.
\end{thm}

\begin{proof}
Substituting \eqref{eq:closedu} into the right hand side of \eqref{sys:CE}
yields the equation \eqref{eq:match1}. Therefore, the continuity
equation matches \eqref{eq:FPK} for $R=\frac{1}{2}hh^{\top}$. The
desired convergence then follows directly from Theorem \ref{thm:main}.
\end{proof}
The control law (\ref{eq:closedu}) consists of two components. The
static state-feedback term $v(x)$ shapes the deterministic drift
into the desired port-Hamiltonian reference drift. The density-feedback
term $-g(x)^{\dagger}R(x)\nabla\ln\rho(t,x)$ emulates the diffusive
part of the stochastic reference model at the PDE level, even though
the closed-loop microscopic dynamics remains deterministic. Although
the controller depends explicitly on the density, it is pointwise
in the microscopic state variable $x$. Therefore, in particle implementations,
$\nabla\ln\rho(t,x)$ can in principle be approximated using nearby
particles. This locality property is important for scalable and decentralized
implementation, which is discussed in the next subsection.
\begin{rem}
A particularly transparent case occurs when ${\rm Im}g(x)=\mathbb{R}^{n}$
for every $x\in\mathbb{R}^{n}$. In this case, for any $H$, one may
choose $J=0$, $R=I$ and $v(x)=g(x)^{-1}-f(x)-\nabla H(x)$. This
results in 
\begin{align*}
\dot{x} & =-\nabla H(x)-\nabla\ln\rho(t,x)\\
 & =-\nabla\left(H(x)+\ln\rho(t,x)\right)\\
 & =-\nabla\ln\frac{\rho}{\pi},
\end{align*}
whose corresponding continuity equation is the FPK equation associated
with the overdamped Langevin dynamics
\[
dx=-\nabla Hdt+\sqrt{2}dW.
\]
\end{rem}

\begin{rem}
In equation \eqref{eq:match2}, the energy function $H$ determines
the target densi\-ty function through $H(x)=-\ln\rho_{\infty}(x)$.
With this choice, the main remaining task is to solve the structure
matching condition \eqref{eq:match2} - \eqref{eq:match3} and to
verify Assumptions H1 - H4. Another possible construction is to first
transform the problem into coordinates in which the desired density
becomes Gaussian. Let $z=\Phi(x)$, and let $\gamma$ denote the density
of $\mathscr{N}(0,I)$. In the $z$ coordinate, $H$ is desired to
be a quadratic function $H(\cdot)=\frac{1}{2}|\cdot|^{2}$. To achieve
this, $\Phi$ should satisfy the following equation
\[
\rho_{\infty}(x)=\gamma(\Phi(x))|\det D\Phi(x)|.
\]
One way to construct such a map $\Phi$ is through the quadratic cost
optimal transport problem,
\begin{align*}
\min_{\pi}\frac{1}{2}\int_{\mathbb{R}^{2n}}|x-y|^{2}{\rm d}\pi(x,y),
\end{align*}
where $\pi$ has marginal distributions $\rho_{\infty}$ and $\mathscr{N}(0,I)$
respectively. The optimal $\pi_{*}$ will be supported on a graph
$(x,T(x))$, from which we set $\Phi=T$. Another possible construction
is to use Rosenblatt transform \cite{rosenblatt1952remarks}.
\end{rem}

\subsection{Implementation of the control law}

The control law \eqref{eq:closedu} consists of two parts of feedback.
The state feedback part is easy to implement, while the second part,
i.e., the density feedback term $\nabla\ln\rho(t,x)$, requires more
effort. This subsection discusses a particle-based implementation
of this density-feedback term.

Assume that the ensemble consists of $N$ agents with states labeled
as $\{x^{j}(t)\}^{N}_{j=1}$. At time $t$, the empirical distribution
of the population is $\mu(t)=\frac{1}{N}\sum^{N}_{j=1}\delta_{x^{j}(t)}$
where $\delta_{x}$ is the Dirac mass. To obtain a smooth approximation
of the density, we replace each Dirac mass by a Gaussian kernel. Let
$G_{\sigma}(x)=(2\pi\sigma^{2})^{-n/2}e^{-\frac{|x|^{2}}{2\sigma^{2}}}$.
Then the kernel density estimator is 
\begin{equation}
\tilde{\rho}(t,x)=\frac{1}{N}\sum^{N}_{j=1}G_{\sigma}(x-x^{j}(t)), \label{eq:rho tilde}
\end{equation}
from which we can compute the gradient
\begin{equation}
\nabla\ln\tilde{\rho}(t,x)=-\frac{\sum^{N}_{j=1}G_{\sigma}(x-x^{j}(t))(x-x^{j}(t))}{\sigma^{2}\sum^{N}_{j=1}G_{\sigma}(x-x^{j}(t))}.\label{eq:drho tilde}
\end{equation}
This gives
a particle approximation of the density-feedback term.

The estimator \eqref{eq:drho tilde} uses global information. By choosing
$\sigma$ small, $G_{\sigma}$ is highly concentrated, and the value
of $\tilde{\rho}(t,x)$ is dominated by particles in a small neighborho\-od
of $x$. Let $B_{\epsilon}(x)$ denote the open ball under certain
metric with center $x$ and radius $\epsilon$, we define the localized
density estimate,
\begin{equation}
\hat{\rho}(t,x):=\frac{1}{N}\sum_{j:\,x^{j}(t)\in B_{\epsilon}(x)}G_{\sigma}(x-x^{j}(t)).\label{eq:rho hat}
\end{equation}
Correspondingly, we approximate $\nabla\ln\hat{\rho}(t,x)$ by
\begin{equation}
\nabla\ln\hat{\rho}(t,x)=-\frac{\sum_{j:\,x^{j}(t)\in B_{\epsilon}(x)}G_{\sigma}(x-x^{j}(t))(x-x^{j}(t))}{\sigma^{2}\sum_{j:\,x^{j}(t)\in B_{\epsilon}(x)}G_{\sigma}(x-x^{j}(t))}.\label{eq:appox  grad ln rho}
\end{equation}
The expression (\ref{eq:appox  grad ln rho}) depends only on particles
in the neighborhood $B_{\epsilon}(x)$. Hence for each agent with
state $x$, the implemented control law becomes
\[
u(t,x)=v(x)-g(x)^{\dagger}R(x)\nabla\ln\hat{\rho}(t,x).
\]
Thus each agent only needs its own state and the states of nearby
agents in order to evaluate the control input. 

In practice, the parameters $\sigma$ and $\epsilon$ determine the
tradeoff between locality and accuracy. A smaller $\ensuremath{\sigma}$
or $\epsilon$ makes the controller more local, but may lead to a
noisy density estimate when the number of neighboring particles is
small. Conversely, larger values give smoother estimates but require
information from a larger neighborhood. If the denominator in \eqref{eq:appox  grad ln rho}
becomes too small, one may use a regularized version
\[
-\frac{\sum_{j:\,x^{j}(t)\in B_{\epsilon}(x)}G_{\sigma}(x-x^{j}(t))(x-x^{j}(t))}{\sigma^{2}(\eta+\sum_{j:\,x^{j}(t)\in B_{\epsilon}(x)}G_{\sigma}(x-x^{j}(t)))}
\]
for some positive number $\eta$ to avoid numerical singularities.

\section{Linear systems and Gaussian terminals}\label{sec:lin-sys}

\subsection{Universality of linear stochastic port-Hamiltonian structure}\label{subsec:uni-pH-lin}

Cons\-ider an ensemble of linear time-invariant systems
\begin{equation}
\dot{x}=Ax+B_{1}u, \label{sys:LTI}
\end{equation}
where $A\in\mathbb{R}^{n\times n}$ and $B_{1}\in\mathbb{R}^{n\times m}$.
The notation $B_{1}$ is used to distinguish the control input matrix
from the diffusion matrix introduced below. Suppose that the ensemble
has a non-Gaussian initial distribution $x_{0}\sim\rho_{0}$. The
goal is to stabilize $\rho_{0}$ to a target Gaussian distribution
$\mathscr{N}(\mu,\Sigma)$ with $\Sigma\succ0$. To streamline the
presentation, we first consider the zero-mean case $\mu=0$. The non-zero
mean case will be discussed later. 

As in the nonlinear setting, the main idea is to match the density
dynamics of \eqref{sys:LTI} to the FPK equation to an adequate stochastic
port-Hamiltonian reference model. In the linear Gaussian setting,
the reference model takes the form,
\begin{equation}
dx=(J-R)\nabla H(x)dt+B_{2}dW, \label{sys:SDE}
\end{equation}
with initial condition $x_{0}\sim\rho_{0}$. Here $J=-J^{\top}$,
$R=R^{\top}\succeq0$ are constant matrices, $H$ is a quadratic function
and $B_{2}\in\mathbb{R}^{n\times p}$ is a constant diffusion matrix.
In the general nonlinear case, the damping and diffusion matrices
are linked by the algebraic relation $R=\frac{1}{2}hh^{\top}$, which
would become $R=\frac{1}{2}B_{2}B^{\top}_{2}$ in the linear setting.
One might expect that this damping-diffusion relation could be relaxed
in the linear setting. The following result shows that this is not
the case: any stable linear diffusion with a nondegenerate Gaussian
invariant distribution admits a stochastic port-Hamiltonian representation
of the form \eqref{sys:SDE} satisfying the structural assumptions
used in Section \ref{sec:OT-determ}. 

Consider the uncontrolled linear stochastic system,
\begin{equation}
dx=Axdt+B_{2}dW.\label{sys:LTI-sto-noinput}
\end{equation}
We have the following universality result:
\begin{thm}
\label{thm:universal}The system \eqref{sys:LTI-sto-noinput} converges
to some non-degenerate Gaussian distr\-ibution $\pi$ for any admissible
positive initial distribution $x_{0}\sim\rho_{0}$ with ${\rm KL}(\rho_{0}\|\pi)<\infty$
if and only if the system can be written in the form \eqref{sys:SDE}
with quadratic energy $H(x)=\frac{1}{2}x^{\top}\Sigma^{-1}x$ with
$\Sigma\succ0$, and the corresponding stochastic port-Hamiltonian
system satisfies Assumptions H1 -- H4.
\end{thm}
\begin{proof}
The sufficiency follows directly from Theorem \ref{thm:match}. We
need only show necessity. Suppose that \eqref{sys:LTI-sto-noinput}
converges to a non-degenerate Gaussian distribution $\mathscr{N}(0,\Sigma)$.
By Theorem \ref{thm:Arnold} given in Appendix, it is necessary that $A$ is Hurwitz,
$(A,B_{2})$ is controllable, and 
\begin{equation}
A\Sigma+\Sigma A^{\top}=-B_{2}B^{\top}_{2}.\label{eq:ctrl-eq}
\end{equation}
Let $H(x)=\frac{1}{2}x^{\top}\Sigma^{-1}x$, and 
\begin{align}
J & =\frac{1}{2}(A\Sigma-\Sigma A^{\top}),\nonumber \\
R & =\frac{1}{2}B_{2}B^{\top}_{2}.\label{eq:R-B2}
\end{align}
It is readily checked that $(J-R)\nabla H(x)=Ax$. 

It remains to verify Assumptions H1 - H4. Assumption H1 follows from
$\Sigma\succ0$. Assumption H2 holds because $J$ is constant and
that \eqref{eq:R-B2} holds. Assumption H3 follows from the Hörmander
condition. Indeed, consider the vector fields
\[
V_{0}=Ax,\;V_{i}=B^{i}_{2},\;i=1,\cdots,p,
\]
where $B^{i}_{2}$ is the $i$-th column of $B_{2}$. The iterated
Lie brackets generated by these vector fields contain the column span
of $[B_{2},AB_{2},\cdots,A^{n-1}B_{2}]$ which has rank $n$. Finally,
Assumption H4 follows from the global convergence. 
\end{proof}

\begin{rem}
Theorem \ref{thm:universal} shows that the stochastic port-Hamiltonian
structure, tog\-ether with Assumptions H1 - H4, is universal for linear
stable mass transportation toward nondegenerate Gaussian distributions.
In particular, the algebraic relation $R=\frac{1}{2}B_{2}B^{\top}_{2}$
is not merely a technical assumption inherited from the nonlinear
theory; it is forced by the covariance balance equation of the limiting
Gaussian distribution.
\end{rem}
We next ask when a controlled linear stochastic system can be shaped
into such a stable linear stochastic port-Hamiltonian form. Consider
\begin{equation}
dx=(Ax+B_{1}u)dt+B_{2}dW.\label{sys:LTI-sto}
\end{equation}
We seek a static state-feedback $u=Kx$ such that the closed-loop system
has a nondegenerate Gaussian invariant distribution and converges
to it from arbitrary admissible positive initial densities.

For linear stochastic systems, one can discuss convergence rate.
We say that the deterministic linear system $\dot{x}=Ax$ has convergence
rate at least $\lambda>0$ if there exists a constant $k>0$, such
that $|x(t)|\le k|x_{0}|e^{-\lambda t}$ for all $t\ge0$ and initial
condition $x_{0}$. The convergence rate is determined by the eigenstructure
of $A$. Equivalently, all eigenvalues of $F$ have real parts no
larger than $-\lambda$, up to the usual polynomial prefactors associated
with nontrivial Jordan blocks.

We are now ready to state the second main result of this section:
\begin{thm}
\label{thm:fromArnold}Consider the system (\ref{sys:LTI-sto}). 
Let $\lambda>0$ be a constant. Then there exist a static state feedback
controller $u=Kx$, a non-degenerate Gaussian distribution $\rho_{\infty}$,
a constant $k>0$, such that the corresponding closed-loop density
$\rho(t,\cdot)$ satisfies 
\[
{\rm KL}(\rho(t,\cdot)\|\rho_{\infty})\le ke^{-\lambda t}{\rm KL}(\rho_{0}\|\rho_{\infty}),
\]
for all $t\ge0$ and all positive initial distributions $\rho_{0}$
, if and only if the following conditions hold: 
\begin{enumerate}
\item $(A,B_{1})$ is stabilizable, and the convergence rate of the non-controllable
subsys\-tem of $\dot{x}=Ax+B_{1}u$ has convergence rate at least $\lambda$;
\item $(A,[B_{1}\;B_{2}])$ is controllable;
\item $B_{2}\neq0$. 
\end{enumerate}
\end{thm}
\begin{proof}
See Appendix.
\end{proof}

\begin{rem}
Theorem \ref{thm:fromArnold} provides an easy-to-check conditions
for the existence of a controlled non-degenerate Gaussian steady state.
The theorem does not imply that an arbitrary prescribed Gaussian distribution
can be reached. It only guarantees the existence of some nondegenerate
Gaussian steady state.
\end{rem}

\begin{rem}
When the desired distribution $\mathscr{N}(0,\Sigma)$ is fixed, one
can in principle characterize all static state-feedback controllers
that render the closed-loop system exponentially convergent to this
steady state in the KL-divergence, using techniques from covariance
control and related Lyapunov-based characterizations \cite{hotz1987covariance}.
However, this direction is not pursued here, since stochastic steering
and controller classification are not the primary focus of the paper.
\end{rem}

\begin{rem}
If the initial density $\rho_{0}$ is already Gaussian, the characterizations
of Theorem \ref{thm:fromArnold} fall within the framework of linear
quadratic theory \cite{anderson2007optimal} since the dissipation
of the KL-divergence can be computed analytically using the covariance
of the system state. Theorem \ref{thm:fromArnold} is more useful
when $\rho_{0}$ is non-Gaussian, since it provides convergen\-ce for
general admissible positive initial densities. This is precisely the
point needed for density feedback rather than merely covariance control.
\end{rem}

\subsection{Mass transportation of deterministic systems}\label{subsec:determ}

We now return to the original problem of distribution steering for
the deterministic linear system \eqref{sys:LTI}. The probability
density associated with the microscopic dynamics satisfies the continuity
equation,
\begin{equation}
\partial_{t}\rho+{\rm div}(\rho(Ax+B_{1}u))=0.\label{eq:CE}
\end{equation}
As in the nonlinear case, the idea is to search for a controller $u(t,x)$
such that \eqref{eq:CE} reproduces the FPK equation of \eqref{sys:SDE}.
More precisely, we seek to match \eqref{eq:CE} with
\begin{equation}
\partial_{t}\rho+{\rm div}(\rho A_{cl}x-\frac{1}{2}B_{2}B^{\top}_{2}\nabla\rho)=0,\label{eq:FP}
\end{equation}
where $A_{cl}x=(J-R)\nabla H(x)$. We therefore consider controllers
of the form
\begin{equation}
u(t,x)=K_{1}x+K_{2}\nabla\ln\rho(t,x), \label{eq:u-K1-K2}
\end{equation}
where $K_{1}$ and $K_{2}$ are chosen so that
\begin{align}
A_{cl} & =A+B_{1}K_{1}\text{ and }-\frac{1}{2}B_{2}B^{\top}_{2}=B_{1}K_{2}.\label{eq:K1K2-equ}
\end{align}

\begin{rem}
The nonzero-mean case can be handled by a translation. Suppose $\mu\in\mathbb{R}^{n}$
is such that there exists some $u_{*}\in\mathbb{R}^{m}$, satisfying
$A\mu+B_{1}u_{*}=0$. Then we can define a new state $\tilde{x}=x-\mu$,
and $\tilde{x}$ satisfies exactly the same equation (\ref{sys:SDE}).
Hence, to stabilize the system to $\mathscr{N}(\mu,\Sigma)$, one
may use
\[
u(t,x)=K_{1}(x-\mu)+u_{*}+K_{2}\nabla\ln\rho(t,x-\mu).
\]
\end{rem}

The following result gives a complete characterization of when a deterministic
linear ensemble can be stabilized to some nondegenerate Gaussian distribution
by density feedback.
\begin{cor}
\label{cor:det-LTI}Consider an ensemble of the system \eqref{sys:LTI}.
The following stateme\-nts are equivalent:
\end{cor}
\begin{enumerate}
\item There exist a non-degenerate Gaussian distribution $\rho_{\infty}$,
some constants $k,\lambda>0$, and a controller $u(t,x)$ of the form
\eqref{eq:u-K1-K2}, such that the closed-loop density satisfies
\begin{equation}
{\rm KL}(\rho(t,\cdot)\|\rho_{\infty})\le ke^{-\lambda t}{\rm KL}(\rho_{0}\|\rho_{\infty}), \label{eq:KL rhot}
\end{equation}
for all $t\ge0$ and positive admissible initial distribution $\rho_{0}$.
\item $(A,B_{1})$ is controllable.
\item There exists a $\Sigma\succ0$ and a matrix $\Theta\in\mathbb{R}^{m\times n}$
such that
\begin{equation}
A\Sigma+\Sigma A^{\top}+B_{1}\Theta+\Theta^{\top}B^{\top}_{1}+B_{1}B^{\top}_{1}=0, \label{eq:ME-1}
\end{equation}
and that $A+B_{1}\Theta\Sigma^{-1}$ is Hurwitz.
\end{enumerate}
Moreover, when these conditions hold, one possible controller is
\begin{equation}
u(t,x)=\Theta\Sigma^{-1}x-\frac{1}{2}B^{\top}_{1}\nabla\ln\rho(t,x).\label{eq:ctrl-determ}
\end{equation}

\begin{proof}
We proceed by showing 1) $\Rightarrow$ 2) $\Rightarrow$ 3) $\Rightarrow$
1). 

1) $\Rightarrow$ 2). Suppose, by contradiction, that $(A,B_{1})$
is uncontrollable. By the controllability decomposition, the system
\eqref{sys:LTI} can be rewritten, after a linear change of coordinates,
as
\[
\begin{bmatrix}\dot{x}_{c}\\
\dot{x}_{u}
\end{bmatrix}=\begin{bmatrix}A_{c} & A_{cu}\\
0 & A_{u}
\end{bmatrix}\begin{bmatrix}x_{c}\\
x_{u}
\end{bmatrix}+\begin{bmatrix}B_{c}\\
0
\end{bmatrix}u.
\]
Now, consider an initial distribution $\rho_{0}(x)=\varphi(x_{u})\psi(x_{c})$
such that both $\varphi$ and $\psi$ are probability distributions.
Obviously, if $\varphi$ is non-Gaussian, $\rho_{0}(x)$ can never
converge to Gaussian since $\dot{x}_{u}=A_{u}x_{u}$. Thus $(A,B_{1})$
is controllable. 

2) $\Rightarrow$ 3). Suppose that $(A,B_{1})$ is controllable. Choose
$K_{1}$ such that $A+B_{1}K_{1}$ is Hurwitz. Since controllability
is invariant under state feedback, $(A+B_{1}K_{1},B_{1})$ is controllable.
By Lemma \ref{lem:AB}, there exists a unique $\Sigma\succ0$ satisfying
\begin{equation}
(A+B_{1}K_{1})\Sigma+\Sigma(A+B_{1}K_{1})^{\top}=-B_{1}B^{\top}_{1}.\label{eq:ME-K1}
\end{equation}
Now, taking $\Theta=K_{1}\Sigma$ gives \eqref{eq:ME-1}

3) $\Rightarrow$ 1). Suppose that $\Sigma\succ0$ and $\Theta$ satisfy
\eqref{eq:ME-1}, and that $A+B_{1}\Theta\Sigma^{-1}$ is Hurwitz.
Substituting the controller \eqref{eq:ctrl-determ} into the continuity
equation \eqref{eq:CE} gives
\begin{align*}
 & \partial_{t}\rho+{\rm div}(\rho(Ax+B_{1}u))\\
= & \partial_{t}\rho+{\rm div}(\rho(A+B_{1}\Theta\Sigma^{-1})x-\frac{1}{2}B_{1}B^{\top}_{1}\nabla\rho)\\
= & 0.
\end{align*}
This is exactly the FPK equation associated with the linear stochastic
system,
\begin{equation}
dx=(A+B_{1}\Theta\Sigma^{-1})xdt+B_{1}dW.\label{sys:det-to-sto}
\end{equation}
In view of \eqref{eq:ME-1} and $A+B_{1}\Theta\Sigma^{-1}$ being
Hurwitz, we deduce that $(A+B_{1}\Theta\Sigma^{-1},B_{1})$ is controllable,
and therefore the conclusion follows immediately from Theorem \ref{thm:Arnold} given in Appendix. 
\end{proof}
\begin{rem}
In view of the reference model \eqref{sys:det-to-sto}, the controller
\eqref{eq:ctrl-determ} correspon\-ds to choosing $K_{2}=-\frac{1}{2}B^{\top}_{1}$
in \eqref{eq:u-K1-K2} and set $B_{2}$ as $B_{1}$. One might wonder
whether we could replace the diffusion term
$B_{1}dW$  in \eqref{sys:det-to-sto} by a more general form $B_{2}dW$, and then relax the controllability
of $(A,B_{1})$ to stabilizability. Corollary \ref{cor:det-LTI} excludes this
possibility. This can also be verified from the following reasoning.
The second algebraic equation in \eqref{eq:K1K2-equ} requires ${\rm Im}B_{2}\subset{\rm Im}B_{1}$.
Now according to Theorem \ref{thm:fromArnold}, the pair $(A,[B_{1},B_{2}])$
must be controllable, this together with ${\rm Im}B_{2}\subset{\rm Im}B_{1}$
implies that $(A,B_{1})$ is controllable. 
\end{rem}
Now, to check if it is possible to stabilize an ensemble of the system
\eqref{sys:LTI} to a given target $\mathscr{N}(0,\Sigma)$, it suffices
to check the solvability of the matrix equation \eqref{eq:ME-1}.
That is, we have the following corollary.
\begin{cor}
\label{cor:det-LTI2}Suppose that $(A,B_{1})$ is controllable. Then
an admissible density of system \eqref{sys:LTI} can be exponentially
stabilized to the given non-degenerate Gaussian distribution $\mathscr{N}(0,\Sigma)$
by a controller of the form \eqref{eq:u-K1-K2} if and only if \eqref{eq:ME-1}
is solvable for $\Theta$. 
\end{cor}
\begin{proof}
Suppose that \eqref{sys:LTI} can be stabilized to $\mathscr{N}(0,\Sigma)$
by a controller of the form \eqref{eq:u-K1-K2}. Then by Theorem \ref{cor:det-LTI},
the equation \eqref{eq:ME-1} holds for some $\Theta$. Conversely,
suppose \eqref{eq:ME-1} holds, which can be written as \eqref{eq:ME-K1}
for $K_{1}=\Theta\Sigma^{-1}$. By assumption, $(A,B_{1})$ is controllable,
as a result, $(A+B_{1}K_{1},B_{1})$ is also controllable. Now the
corollary follows from Lemma \ref{lem:AB}. 
\end{proof}

\subsection{Covariance stabilization}\label{subsec:cov-steer}

We now examine the controller (\ref{eq:ctrl-determ}) in the special
case where the initial density $\rho_{0}$ is Gaussian. Assume $\rho_{0}\sim\mathscr{N}(0,\Sigma_{0})$.
Let $A_{cl}=A+B_{1}\Theta\Sigma^{-1}$ be Hurwitz. From (\ref{sys:det-to-sto}),
the density $\rho(t,\cdot)$ remains zero mean Gaussian for all $t$.
That is, $\rho(t,x)\sim\mathscr{N}(0,P(t))$ and hence $\nabla\ln\rho(t,x)=-P^{-1}(t)x$
where $P(t)$ satisfies the following matrix differential equation,
\begin{equation}
\dot{P}=A_{cl}P+PA^{\top}_{cl}+B_{1}B^{\top}_{1},\quad P(0)=\Sigma_{0}.\label{eq:cov-ode}
\end{equation}
The controller is
\begin{align}
u(t,x) & =\left(K+\frac{1}{2}B^{\top}_{1}P^{-1}(t)\right)x.\label{eq: u-det KP}
\end{align}
Thus, in the Gaussian initial case, the density-feedback controller
reduces to a purely time-varying state-feedback law. However, it is
worth mentioning that although (\ref{eq:ctrl-determ}) is local in
the microscopic state $x$, the controller in the form (\ref{eq: u-det KP})
depends on the global information $P(t)$.

A natural question is what the controller (\ref{eq: u-det KP}) does
when the initial condition is not Gaussian. In fact, this corresponds
precisely to covariance control \cite{hotz1987covariance,Hotz1985}.
Even if $\rho_{0}$ is non-Gaussian, the covariance of $x(t)$ still
evolves according to (\ref{eq:cov-ode}). Since $A_{cl}$ is Hurwitzt,
and that $(A_{cl},B_{1})$ is controllable, the covariance $P(t)$
converges exponentially to the unique positive definite solution $\Sigma$
of (\ref{eq:ME-1}).
\begin{rem}
Alternatively, to avoid computing the inverse $P^{-1}(t)$
in \eqref{eq: u-det KP}, one could instead solve the Riccati equation,
\[
\dot{Q}=-QA_{cl}-A^{\top}_{cl}Q-QB_{1}B^{\top}_{1}Q,\;Q(0)=\Sigma^{-1}_{0},
\]
and apply the controller $u(t,x)=\left(K+\frac{1}{2}B^{\top}_{1}Q(t)\right)x$.
However, solving the Lyapunov equation \eqref{eq:cov-ode} may be
numerically more stable than directly integrating the Riccati equation
for $Q(t)$. 
\end{rem}

\section{Nonlinear examples}\label{sec:NL-exmp}

\subsection{The kinetic FPK model}

We first consider a double-integrator ensemble
\begin{align*}
\dot{x} & =v\\
\dot{v} & =u
\end{align*}
where $x,v\in\mathbb{R}^{n}$. The goal is to steer the density of
the ensemble to the Gibbs density $\pi(x,v)=ce^{-\frac{1}{2}|v|^{2}-V(x)}$
where $c$ is a normalizing constant. To achieve that, we consider
\begin{equation}
\begin{bmatrix}dx\\
dv
\end{bmatrix}=\begin{bmatrix}0 & I\\
-I & -\gamma I
\end{bmatrix}\nabla Hdt+\begin{bmatrix}0\\
\sqrt{2\gamma}
\end{bmatrix}dW, \label{eq:kinetic FPK}
\end{equation}
where $H=\frac{1}{2}|v|^{2}+V(x)$ and $\gamma>0$. Here $H$ stands
for the mechanical energy of the system. The FPK equation associated
with \eqref{eq:kinetic FPK} is called the kinetic FPK equation and
it is known that an admissible initial density converges asymptotically
to the steady state $\pi$ \cite{villani2009optimal}, and as a consequence,
the ensemble of the double integrator converges to $\pi$ under the
control law
\begin{align}
u & =-\nabla V(x)-\gamma v-\gamma\frac{\partial\ln\rho(x,v)}{\partial v}.\label{eq:kinetic-u}
\end{align}

Alternatively, we verify the above conclusion using Theorem \ref{thm:main}.

It is obvious that whenever the steady state density $\pi(x,v)$ is
well defined, Assumption H1 and H2 are trivially satisfied. To see
that it satisfies Assumption H3, we verify the Hörmander condition.
Let $V_{0}={\rm col}(v,-\nabla V(x)-\gamma v)$, $V_{i}={\rm col}(0_{n},\sqrt{2\gamma}e_{i})$,
where $e_{i}$ stands for the $i$-th coordinate vector. Then 
\[
[V_{0},V_{i}]=-\begin{bmatrix}0 & I\\
-\nabla^{2}V & -\gamma I
\end{bmatrix}\begin{bmatrix}0\\
\sqrt{2\gamma}e_{i}
\end{bmatrix}=\begin{bmatrix}\sqrt{2\gamma}e_{i}\\
-\gamma\sqrt{2\gamma}e_{i}
\end{bmatrix},
\]
from which it follows that 
\[
{\rm span}_{\mathbb{R}}\{V_{0},V_{1},\cdots,V_{n},[V_{0},V_{1}],\cdots,[V_{0},V_{n}]\}=\mathbb{R}^{2n}.
\]
For Assumption H4, set $h=[0,\sqrt{2\gamma}]^{\top}$, we have
\begin{align*}
h^{\top}\nabla\ln\frac{\rho}{\pi} & =0\Leftrightarrow\nabla_{v}\left(\ln\rho+\frac{1}{2}|v|^{2}+V(x)\right)=0\\
 & \Leftrightarrow\ln\rho=-\frac{1}{2}|v|^{2}+\varphi(t,x)
\end{align*}
for some smooth function $\varphi$. Substituting this into the FPK
equation $\partial_{t}\rho+v\cdot\nabla_{x}\rho-\nabla_{x}V\cdot\nabla_{v}\rho=\gamma\nabla_{v}\cdot(\nabla_{v}\rho+v\rho)$
of \eqref{eq:kinetic FPK}
gives
\[
\partial_{t}\varphi=-v\cdot\nabla_{x}(\varphi+V).
\]
Since $\varphi$ does not dependent on $v$, it follows that $\varphi(t,x)+V(x)=\eta(t)$
for some $\eta$ and $\partial_{t}\eta=0$. Hence, $\eta$ is a constant
and that $\rho(t,x)\propto e^{-\frac{1}{2}|v|^{2}-V(x)}$. That is,
the largest invariant set in $\mathscr{I}$ is the singleton $\{\pi\}$.
Assumptions H1 - H4 are all satisfied.

\subsection{A 3D example}

Consider the three dimensional nonlinear system,
\begin{align}
\dot{x}_{1} & =x_{2}\nonumber \\
\dot{x}_{2} & =-x_{1}-\alpha(x_{1})x_{3}\label{eq:sys_NL-examp}\\
\dot{x}_{3} & =u\nonumber 
\end{align}
where $\alpha:\mathbb{R}\to\mathbb{R}$ is a $C^{1}$ function which
is either strictly negative or strictly positive on $\mathbb{R}$.
The objective is to steer any admissible positive initial density
to the standard Gaussian distribution $\mathscr{N}(0,I_{3})$.

To use the port-Hamiltonian structure to stabilize the Gaussian distribution,
we choose $H(x)=\frac{1}{2}|x|^{2}$. The matching condition requires
\[
\begin{bmatrix}x_{2}\\
-x_{1}-\alpha(x_{1})x_{3}\\
v
\end{bmatrix}=(J(x)-R(x))\begin{bmatrix}x_{1}\\
x_{2}\\
x_{3}
\end{bmatrix},
\]
together with ${\rm Im}R\subset{\rm Im}\{[0,0,1]^{\top}\}$. Consider
\begin{equation}
J(x)=\begin{bmatrix}0 & 1 & 0\\
-1 & 0 & -\alpha(x_{1})\\
0 & \alpha(x_{1}) & 0
\end{bmatrix},\; R=\begin{bmatrix}0\\
 & 0\\
 &  & 1
\end{bmatrix}. \label{eq:J-R-examp}
\end{equation}
and $v(x)=\alpha(x_{1})x_{2}-x_{3}$. The complete controller should
then take the form 
\begin{equation}
u(t,x)=v(x)-\frac{\partial\ln\rho(t,x)}{\partial x_{3}}.\label{eq:u-NL_exmp}
\end{equation}

We verify that with this controller, the Assumptions H1 - H4 are satisfied. 

Assumption H1 is trivial. For Assumption H2, take $h=[0,0,\sqrt{2}]^{\top}$
so that $R=\frac{1}{2}hh^{\top}$. Moreover, $J$ is divergence-free.
Indeed, $J$ has the form \eqref{eq:div-free J} for $\phi(x)=-x_{3}+\int^{x_{1}}_{0}\alpha(s){\rm d}s$. 

For Assumption H3, we verify the Hörmander condition. Let {\small
\[
V_{0}(x)=\begin{bmatrix}x_{2}\\
-x_{1}-\alpha(x_{1})x_{3}\\
\alpha(x_{1})x_{2}-x_{3}
\end{bmatrix},\quad V_{1}(x)=\begin{bmatrix}0\\
0\\
1
\end{bmatrix}.
\]
\begin{align*}
V_{2} & =[V_{0},V_{1}]=\begin{bmatrix}0\\
\alpha(x_{1})\\
1
\end{bmatrix},V_{3}=[V_{0},V_{2}]=\begin{bmatrix}-\alpha(x_{1})\\
\alpha'(x_{1})x_{2}+\alpha(x_{1})\\
1-\alpha(x_{1})^{2}.
\end{bmatrix}
\end{align*}
}At every $x$, 
\begin{align*}
{\rm span}_{\mathbb{R}}\left\{ V_{1},V_{2},V_{3}\right\}  & ={\rm span}_{\mathbb{R}}\left\{ \begin{bmatrix}0\\
0\\
1
\end{bmatrix},\begin{bmatrix}0\\
\alpha(x_{1})\\
0
\end{bmatrix},\begin{bmatrix}\alpha(x_{1})\\
0\\
0
\end{bmatrix}\right\} \\
 & =\mathbb{R}^{3},
\end{align*}
since $\alpha$ is non-zero everywhere.

Finally, for Assumption H4, note
\begin{align*}
\mathscr{I}= & \left\{ \rho\in\mathscr{P}(\mathbb{R}^{n}):h(\cdot)^{\top}\nabla\log\frac{\rho(\cdot)}{\pi(\cdot)}=0\right\} \\
= & \left\{ \rho\in\mathscr{P}(\mathbb{R}^{n}):\frac{\partial\ln\rho}{\partial x_{3}}=-x_{3}\right\} \\
= & \left\{ \rho:\ln\rho(x)=-\frac{1}{2}x^{2}_{3}+\psi(x_{1},x_{2}),\;\int_{\mathbb{R}^{n}}\rho{\rm d}x=1\right\} .
\end{align*}
Recall the FPK equation
\[
\partial_{t}\rho=-\nabla\cdot\{\rho(J-R)\nabla H)-R\nabla\ln\rho\}.
\]
On the set $\mathscr{I}$, the term $-R\nabla\ln\rho$ is simply $x_{3}$.
Hence the FPK equation reduces to 
\[
\partial_{t}\rho=-\nabla\cdot\rho\begin{bmatrix}x_{2}\\
-x_{1}-\alpha(x_{1})x_{3}\\
\alpha(x_{1})x_{2}
\end{bmatrix}.
\]
Plugging $\ln\rho(t,x)=-\frac{1}{2}x^{2}_{3}+\psi(t,x_{1},x_{2})$
into the above equation, we get $\partial_{t}\psi=x_{1}\partial_{x_{2}}\psi-x_{2}\partial_{x_{1}}\psi+x_{3}(\alpha(x_{1})\partial_{x_{2}}\psi+\alpha(x_{1})x_{2})$.
Since $\psi$ does not depend on $x_{3}$, we deduce that
\begin{align}
\alpha(x_{1})(\partial_{x_{2}}\psi+x_{2}) & =0, \label{eq:alpha}
\end{align}
and 
\begin{equation}
\partial_{t}\psi=x_{1}\partial_{x_{2}}\psi-x_{2}\partial_{x_{1}}\psi.\label{eq:transport-eq}
\end{equation}
It should be noted that we need the strict positivity of $\rho(t,\cdot)$ here, which is guaranteed
by Assumption H3. From \eqref{eq:alpha}, $\partial_{x_{2}}\psi+x_{2}=0$
or $\psi(t,x_{1},x_{2})=-\frac{1}{2}x^{2}_{2}+\eta(t,x_{1})$ for
some $\eta$. Plugging $\psi$ into \eqref{eq:transport-eq} yields
$\partial_{t}\eta=-x_{2}(x_{1}+\partial_{x_{1}}\eta)$. Since $\eta$
is independent of $x_{2}$, we must have $\partial_{t}\eta=\partial_{x_{1}}\eta+x_{1}=0$
from which we obtain $\eta=-\frac{1}{2}x^{2}_{1}+C$ for some time-independent
constant $C$. Putting these together, we get
\[
\ln\rho(t,x_{1},x_{2},x_{3})=-\frac{1}{2}|x|^{2}+C,
\]
which is exactly $\pi$. Thus, the largest invariant set in $\mathscr{I}$
is the singleton $\pi$.

We arrive at the following proposition:
\begin{prop}
Consider the system \eqref{sys:NL} describing an ensemble. Assume
that $\alpha\in C^{1}$ is strictly positive or strictly negative
on $\mathbb{R}$. Then, for every admissible positive smooth initial
density $\rho_{0}$ with ${\rm KL}(\rho_{0}\|\mathscr{N}(0,I_{3}))<\infty$,
the controller \eqref{eq:u-NL_exmp} steers $\rho_{0}$ to $\mathscr{N}(0,I_{3})$
asymptotically in the narrow topology.
\end{prop}

\subsection{Unicycle}\label{subsec:unicycle}

Consider the unicycle model,
\begin{align}
\dot{x} & =u_{1}\cos\theta, \nonumber \\
\dot{y} & =u_{1}\sin\theta, \label{sys:unicycle}\\
\dot{\theta} & =u_{2}, \nonumber 
\end{align}
whose state evolves in $\mathbb{R}^{2}\times S^{1}$, where $S^1$ is the unit circle.

We aim to stabilize the $xy$-marginal of an ensemble of \eqref{sys:unicycle}
to the standard Gaussian distribution. To that end, we choose $H(x,y,\theta)=\frac{1}{2}(x^{2}+y^{2})+\varphi(\theta)$
and stabilize the system to the steady state,
\begin{equation}
\pi(x,y,\theta)=\frac{e^{-H(x,y,\theta)}}{\int_{\mathbb{R}^{2}\times S^{1}}e^{-H(x,y,\theta)}dxdyd\theta}, \label{eq:pi_unicycle}
\end{equation}
with $\varphi(\theta)$ chosen as some smooth function. In the following
we take $\varphi(\theta)=1-\cos\theta$.

By choosing 
\[
v(x,y,\theta)=\begin{bmatrix}-(x\cos\theta+y\sin\theta)\\
-\sin\theta
\end{bmatrix},
\]
the matching condition \eqref{eq:match2} holds for $J=0$ and
\[
R(x,y,\theta)=\begin{bmatrix}(\cos\theta)^{2} & \sin\theta\cos\theta & 0\\
\sin\theta\cos\theta & (\sin\theta)^{2} & 0\\
0 & 0 & 1
\end{bmatrix}.
\]

Theorem \ref{thm:match} then suggests the following control law:
\begin{equation}
u(x,y,\theta)=\begin{bmatrix}-(x\cos\theta+y\sin\theta)\\
-\sin\theta
\end{bmatrix}-\begin{bmatrix}\cos\theta & \sin\theta & 0\\
0 & 0 & 1
\end{bmatrix}\nabla\ln\rho. \label{eq:u-unicycle}
\end{equation}

We now verify the structural assumptions. Assumption H1 holds because
$e^{-H}$ is integrable on $\mathbb{R}^{2}\times S^{1}$. Assumption
H2 also holds since $J=0$, and with 
\[
h(x,y,\theta)=\sqrt{2}\begin{bmatrix}\cos\theta & \sin\theta & 0\\
0 & 0 & 1
\end{bmatrix},
\]
we have $R=\frac{1}{2}hh^{\top}$. For Assumption H3, let 
\[
V_{1}=\begin{bmatrix}\cos\theta\\
\sin\theta\\
0
\end{bmatrix},\;V_{2}=\begin{bmatrix}0\\
0\\
1
\end{bmatrix}.
\]
It is well known that ${\rm span}_{\mathbb{R}}\{V_{1},V_{2},[V_{0},V_{1}]\}=\mathbb{R}^{3}$
at every point. Indeed, we have $[V_{0},V_{1}]=[-\sin\theta,\cos\theta,0]^{\top}$. 

It remains to check Assumption H4. Note that
\[
\begin{bmatrix}\cos\theta & \sin\theta & 0\\
0 & 0 & 1
\end{bmatrix}\left(\nabla\ln\rho+\begin{bmatrix}x\\
y\\
\sin\theta
\end{bmatrix}\right)=0,
\]
is equivalent to 
\begin{align*}
\cos\theta(x+\partial_{1}\ln\rho)+\sin\theta(y+\partial_{2}\ln\rho) & =0,\\
\partial_{3}\ln\rho+\sin\theta & =0.
\end{align*}
From the second equation, we deduce $\ln\rho=\cos\theta+\psi(x,y)$.
Substituting this into the first equation, we obtain $\cos\theta(x+\partial_{1}\psi)+\sin\theta(y+\partial_{2}\psi)=0$
from which we deduce $x+\partial_{1}\psi=y+\partial_{2}\psi=0$ or
\[
\rho(x,y,\theta)=c\exp\left(-\frac{1}{2}x^{2}-\frac{1}{2}y^{2}+\cos\theta\right),
\]
which is exactly the steady state.

We have proved the following proposition.
\begin{prop}
\label{prop:unicycle}Consider the unicycle ensemble \eqref{sys:unicycle}
For every admissible posit\-ive smooth initial density $\rho_{0}$ with
${\rm KL}(\rho_{0}\|\pi)<\infty$ with $\pi$ defined as in \eqref{eq:pi_unicycle},
the controller \eqref{eq:u-unicycle} steers the density to $\pi$
asymptotically in the narrow topology. In particular, the $\ensuremath{(x,y)}$-marginal
converges to the standard Gaussian distribution on $\mathbb{R}^{2}$.
\end{prop}

\section{Simulations}\label{sec:simu}

\begin{example}
Consider the double integrator 
\begin{align*}
\dot{x} & =y,\\
\dot{y} & =u.
\end{align*}
Suppose that the desired distribution is a Gaussian mixture,
\[
\rho_{\infty}=0.2\mathscr{N}(m_{1},\Sigma_{1})+0.8\mathscr{N}(m_{2},\Sigma_{2}),
\]
where $m_{1}=[1,0]^{\top},\;m_{2}=[-1,0]^{\top},\;\Sigma_{1}={\rm diag}(0.1^{2},0.3^{2}),$
and $\Sigma_{2}={\rm diag}(0.5^{2},0.3^{2})$. Consider the energy
function,
\[
H(x,y)=-\ln(0.2\mathscr{N}(m_{1},\Sigma_{1})+0.8\mathscr{N}(m_{2},\Sigma_{2})).
\]
Let $a(x,y)=0.2\mathscr{N}(m_{1},\Sigma_{1})$ and $b(x,y)=0.8\mathscr{N}(m_{2},\Sigma_{2})$.
One can verify that the controller 
\[
v(x,y)=\frac{9(x-1)a(x,y)+0.36(x+1)b(x,y)}{a(x,y)+b(x,y)}-r\frac{y}{0.09}
\]
shapes the double integrator into 
\[
\begin{bmatrix}\dot{x}\\
\dot{y}
\end{bmatrix}=\begin{bmatrix}0 & -0.09\\
0.09 & -r
\end{bmatrix}\nabla H(x,y).
\]
Therefore, we propose the density based feedback controller as 
\[
u(x,y,t)=v(x,y)-r\frac{\partial\ln\rho(x,y,t)}{\partial y}.
\]
Take for example, $r=0.2$. The simulation is shown in Fig. \ref{fig:double_int}, which shows that the system
converges toward the desired non-Gaussian distribution in a few seconds.

\end{example}

\begin{figure}[ht]
	\centering
\includegraphics[scale=0.83]{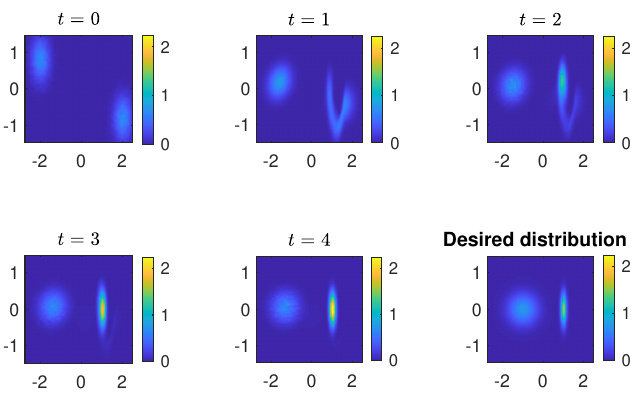}
\caption{Double integrator and Gaussian mixture teminal. }\label{fig:double_int}
\end{figure}

\begin{example}
By Corollarys \ref{cor:det-LTI} and \ref{cor:det-LTI2}, a linear
ensemble of \eqref{sys:LTI} can be stabilized to a given Gaussian
distribution $\mathscr{N}(0,\Sigma)$ if and only if the matrix equation
\eqref{eq:ME-1} has a solution $\Theta$. However, in practice, the
matrix equation may not have a solution when $B_{1}$ has very low
rank. A more realistic objective is then to seek a ``least-square
approximation''. 
\begin{align*}
\min_{\Gamma,\,\Theta}\quad\quad & \|\Gamma-\Sigma\|^{2}_{F}\\
\text{subject to: } & A\Gamma+\Gamma A^{\top}+B_{1}\Theta+\Theta^{\top}B_{1}+B_{1}B^{\top}_{1}=0,\\
 & \Gamma\succ0.
\end{align*}
Then we apply the controller $u(t,x)=\Theta\Gamma^{-1}x-\frac{1}{2}B^{\top}_{1}\nabla\ln\rho(t,x)$.
In this example, we choose $A=\begin{bmatrix}1 & -2\\
0 & 0
\end{bmatrix}$, $B_{1}=\begin{bmatrix}0\\
1
\end{bmatrix}$ and the desired $\Sigma=\begin{bmatrix}1 & 0.4\\
0.4 & 0.5
\end{bmatrix}$. It can be verified that the equation \eqref{eq:ME-1} has no solution.
As a compromise, we use the above optimization program to find a ``lease-square
approximation''. The result is shown in Fig. \ref{fig:lst_LTI}. We
can see that the steady state is relatively close to the desired distribution.

\end{example}
\begin{figure}[ht]
	\centering
\includegraphics[scale=0.9]{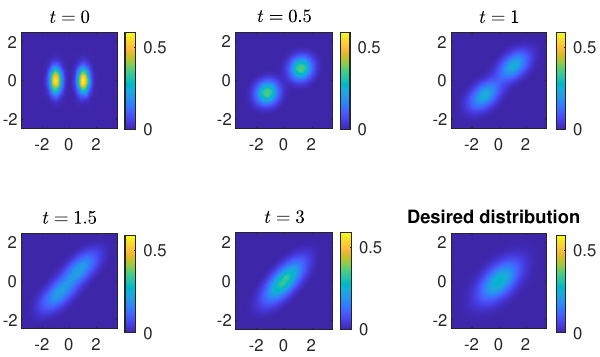}

\caption{Evolution of a ``least-square'' solution. }\label{fig:lst_LTI}
\end{figure}

\begin{example}
We simulate the 3D nonlinear example studied in Section \ref{subsec:unicycle}.
The initial distribution is taken as a Gaussian mixture $0.5\mathscr{N}(\mu_{1},\Sigma)+0.5\mathscr{N}(\mu_{2},\Sigma)$
with $\mu_{1}=[1,0,0]^{\top}$, $\mu_{2}=[-1,0,1]^{\top}$ and $\Sigma={\rm diag}(0.28^{2},0.4^{2},0.28^{2})$.
To visualize the evolution of the density, we plot the evolution of
the two-dimensional marginal distrib\-utions of $\rho(t,x_{1},x_{2},x_{3})$
in Fig. \ref{fig:NL_3D}. The simulation shows that the initial state converges to a Gaussian distribution in a few seconds.

\begin{figure}
\begin{centering}
\includegraphics[scale=0.6]{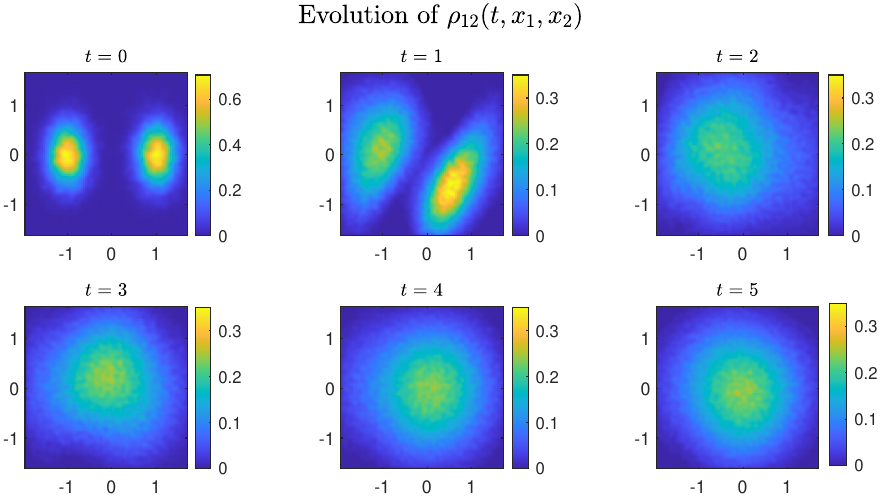}
\par\end{centering}
\vspace{1em}
\begin{centering}
\includegraphics[scale=0.6]{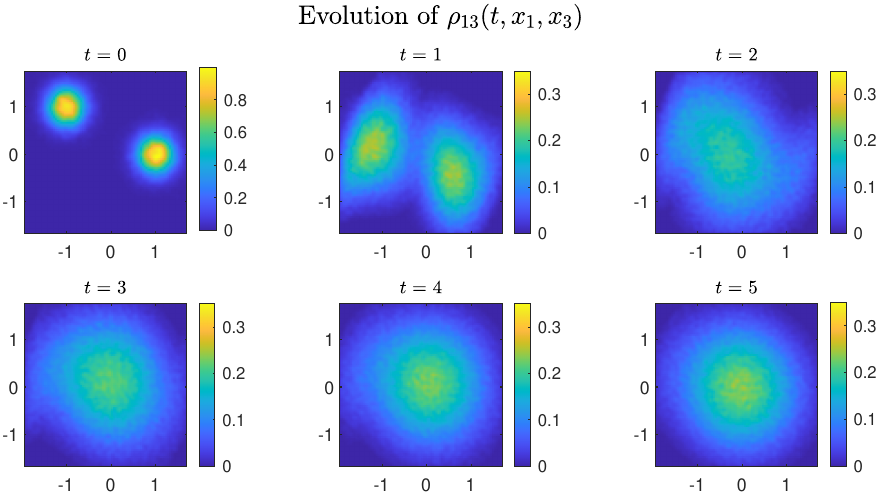}
\par\end{centering}
\vspace{1em}
\begin{centering}
\includegraphics[scale=0.6]{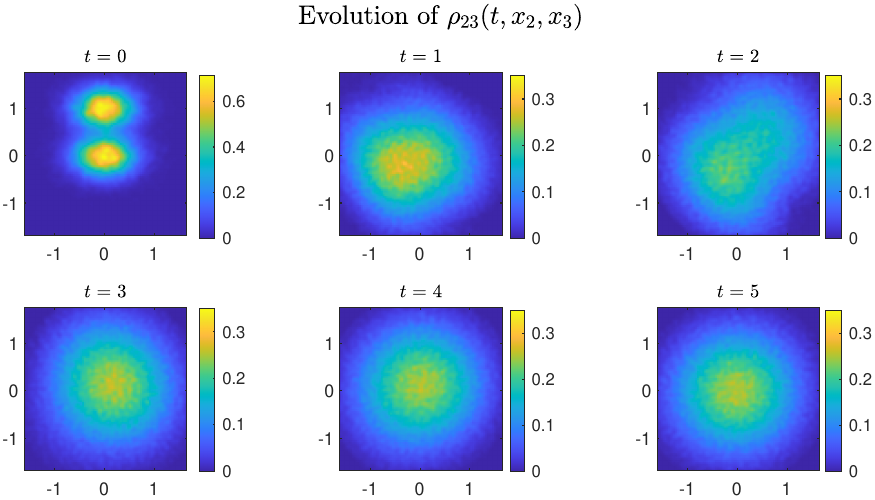}
\par\end{centering}
\caption{The evolution of an ensemble of \eqref{eq:sys_NL-examp} under the
control law \eqref{eq:u-NL_exmp}.}\label{fig:NL_3D}
\end{figure}
\end{example}

\section{Conclusion}

In this paper, we have developed an infinite-horizon mass trans\-portation
framework for nonlinear systems. The approach is inspired by stochastic
and dissipative systems: by emulating the steady-state behavior of
a suitably chosen stochastic port-Hamiltonian system, we obtain
decentralized feedback laws that act locally at the particle level
yet induce global distributional stabilization. This leads to a feedback-based
alternative to finite-horizon optimal transport, with local and decentralized
implementation through particle-based density estimation. For linear
systems with Gaussian target distributions, we further derive necessary
and sufficient conditions for stabilizing density feedback and provide
explicit controller constructio\-ns. Nonlinear examples and numerical
simulations demonstrate the applicability and effectiveness of the
proposed framework.

\section{Appendix}

\setcounter{thm}{0}
\setcounter{lem}{0}
\setcounter{defn}{0}
\setcounter{rem}{0}

\renewcommand{\thethm}{A\arabic{thm}}
\renewcommand{\thelem}{A\arabic{lem}}
\renewcommand{\therem}{A\arabic{rem}}
\renewcommand{\thedefn}{A\arabic{defn}}

\begin{defn}[\cite{haraux1991systemes}]
\label{def:System-Dyn}Consider a metrizable topological space $X$.
A dynamical system on $X$ is a one parameter family of mappings $\{S_{t}\}_{t\ge0}$
on $X$, satisfying the following conditions:

(i) $S_{t}:X\to X$ is continuous for every $t\ge0$;

(ii) $S_{0}$ is the identify map on $X$;

(iii) $S_{t+s}=S_{t}\circ S_{s}$, $\forall s,t\ge0$;

(iv) For any $x\in X$, $t\mapsto S_{t}x$ is continuous.

Let $x\in X$, then 
\[
\omega(x)=\bigcap_{s>0}\overline{\bigcup_{t\ge s}\{S_{t}x\}}
\]
is called the $\omega$-limit set of $x$. 
\end{defn}

\begin{thm}[\cite{haraux1991systemes}]
\label{thm:Lasalle}Consider a dynamical system $\{S(t)\}_{t\ge0}$
as in Definition \ref{def:System-Dyn}. Let $x\in X$, and suppose
$\cup_{t\ge0}\{S(t)x\}$ is precompact in $X$. Then $\omega(x)$
is a compact connected set in $X$ and that $S(t)x$ converges asymptotically
to $\omega(z)$.
\end{thm}
\begin{thm}[\cite{Arnold2014}]
\label{thm:Arnold}Let the initial density $\rho_{0}$ be Lebesgue
measurable, and $D=D^{\top}\in\mathbb{R}^{n\times n}$ be a positive
semi-definite matrix. Let $\lambda>0$ be a positive constant. Then
there exist a probability density $\rho_{\infty}$ and a constant
$k>0$ such that the solution $\rho(t,\cdot)$ of the following Fokker-Planck
equation on $\mathbb{R}_{>0}\times\mathbb{R}^{n}$ 
\[
\partial_{t}\rho={\rm div}(\rho Ax+D\nabla_{x}\rho),
\]
satisfies 
\[
{\rm KL}(\rho(t,\cdot)\|\rho_{\infty})\le ke^{-\lambda t}{\rm KL}(\rho_{0}\|\rho_{\infty}),
\]
for all $t\ge0$ and initial probability densities $\rho_{0}$ if
and only if the following two conditions hold: 
\begin{enumerate}
\item[{{{[}C1{]}}}] There is no non-trivial $A^{\top}$-invariant subspace of $\ker D$, and
\item[{{{[}C2{]}}}] $-A$ is Hurwitz, and the convergence rate of $\dot{x}=-Ax$ is no
less than $\lambda$. 
\end{enumerate}
Moreover, the steady state is of the Gaussian form $\mathscr{N}(0,\Sigma)$,
in which $\Sigma$ is the unique positive definite solution to 
\[
A\Sigma+\Sigma A^{\top}=2D.
\]
\end{thm}
\begin{rem}
Condition C1 is equivalent to $(A,D)$ being controllable. 
\end{rem}
\begin{lem}
\label{lem:AB}The matrix equation 
\[
A\Sigma+\Sigma A^{\top}=-BB^{\top},
\]
has a unique positive definite solution $\Sigma$ if and only if $A$
is Hurwitz and $(A,B)$ is controllable. 
\end{lem}

\begin{proof}[Proof of Theorem \ref{thm:fromArnold}]
By Theorem \ref{thm:Arnold}, the system \eqref{sys:LTI-sto}
can be controlled by static state-feedback $u=Kx$ to a non-degenerate Gaussian
distribution at steady state if and only if $A+B_{1}K$ is Hurwitz
and $(A+B_{1}K,B_{2})$ is controllable. Therefore, it is sufficient
to show that: there exists a matrix $K\in\mathbb{R}^{m\times n}$
such that $A+B_{1}K$ is Hurwitz and $(A+B_{1}K,B_{2})$ is controllable
if and only if the following conditions hold:
\begin{enumerate}
\item $(A,B_{1})$ is stabilizable, 
\item $(A,[B_{1}\;B_{2}])$ is controllable, 
\item $B_{2}\neq0$. 
\end{enumerate}

\emph{Necessity}. Suppose there exists $K$ such that $A+B_{1}K$
is Hurwitz and $(A+B_{1}K,B_{2})$ is controllable. Then it is obvious
that $(A,B_{1})$ is stabilizable and $B_{2}\neq0$. $(A,[B_{1}\;B_{2}])$
being controllable can be deduced from the following inequality: 
\begin{align*}
 & {\rm rank}\begin{bmatrix}sI-A-B_{1}K, & B_{2}\end{bmatrix}\\
\le & {\rm rank}\begin{bmatrix}sI-A-B_{1}K, & B_{1}, & B_{2}\end{bmatrix}\\
= & {\rm rank}\begin{bmatrix}sI-A, & B_{1}, & B_{2}\end{bmatrix}.
\end{align*}

\emph{Sufficiency.} If $(A,B_{1})$ is controllable. For a fixed $K$,
let $p^{\top}_{c}$ be a left eigenvector of $A+B_{1}K$ associated
with the eigenvalue $\lambda$. Thanks to controllability, we can
place the eigenstructure freely, and thus it is easy to achieve $p^{\top}_{c}B_{2}\ne0$
for all left eigenvectors $p^{\top}_{c}$ of $A+B_{1}K$ while making
the latter Hurwitz since $B_{2}\ne0$. This leads to sufficiency
when $(A,B_{1})$ is controllable.

If $(A,B_{1})$ is not fully controllable, to which we can apply controllability
decompos\-ition -- without loss of generality, we may assume $A$,
$B_{1}$, and $B_{2}$ take the following form,
\[
A=\begin{bmatrix}A_{c} & A_{cu}\\
0 & A_{u}
\end{bmatrix},\quad B_{1}=\begin{bmatrix}B_{1c}\\
0
\end{bmatrix},\quad B_{2}=\begin{bmatrix}B_{2c}\\
B_{2u}
\end{bmatrix},
\]
where: 
\begin{itemize}
\item $(A_{c},B_{1c})$ is controllable ($A_{c}\in\mathbb{R}^{n_{c}\times n_{c}}$,
$B_{1c}\in\mathbb{R}^{n_{c}\times m}$), 
\item $A_{u}$ is Hurwitz ($A_{u}\in\mathbb{R}^{n_{u}\times n_{u}}$). 
\end{itemize}
It follows that $(A_{u},B_{2u})$ is controllable. In fact, for any
left eigenvector $v^{\top}_{u}$ of $A_{u}$ ($v^{\top}_{u}A_{u}=\lambda v^{\top}_{u}$),
take $v^{\top}=[0\;v^{\top}_{u}]$. Then $v^{\top}A=\lambda v^{\top}$,
$v^{\top}B_{1}=0$. By controllability of $(A,[B_{1},B_{2}])$, there
holds $v^{\top}B_{2}=v^{\top}_{u}B_{2u}\neq0$. Hence, $(A_{u},B_{2u})$
is controllable.

Let $K=[K_{c}\;K_{u}]$. Choose $K_{c}$ such that the eigenvalues
of $A_{c}+B_{1c}K_{c}$ are all real, negative, and do not overlap
with $\sigma(A_{u})$.

Below, we proceed to construct $K_{u}$ to ensure the controllability
of $(A+B_{1}K,B_{2})$. 

Suppose there exists a row vector $[p^{\top}_{c},q^{\top}_{u}]$ satisfying
\begin{equation}
[p^{\top}_{c},q^{\top}_{u}]\begin{bmatrix}\lambda I-A-B_{1}K, & B_{2}\end{bmatrix}=0.\label{eq:pc-qu}
\end{equation}
Expressing the relation in terms of the controllable decomposition
gives 
\begin{align}
p^{\top}_{c}(\lambda I_{n_{c}}-(A_{c}+B_{1c}K_{c})) & =0,\label{eq1}\\
-p^{\top}_{c}(A_{cu}+B_{1c}K_{u})+q^{\top}_{u}(\lambda I_{n_{u}}-A_{u}) & =0,\label{eq2}\\
p^{\top}_{c}B_{2c}+q^{\top}_{u}B_{2u} & =0.\label{eq3}
\end{align}
We show that there exists a real constant matrix $K_{u}$ such that
the preceding equations imply $[p^{\top}_{c},q^{\top}_{u}]=0$.

\emph{Case 1:} $\lambda\notin\sigma(A_{c}+B_{1c}K_{c})$. From (\ref{eq1})
we deduce $p^{\top}_{c}=0$ and as a consequence $q^{\top}_{u}(\lambda I_{n_{u}}-A_{u})=0$
and $q^{\top}_{u}B_{2u}=0$. Since $(A_{u},B_{2u})$ is controllable,
this implies $q^{\top}_{u}=0$. Thus $[p^{\top}_{c},q^{\top}_{u}]=0$.

\emph{Case 2:} $\lambda\in\sigma(A_{c}+B_{1c}K_{c})$. We shall show
that $p^{\top}_{c}=0$. Otherwise, $p^{\top}_{c}$ is a left eigenvector
of $A_{c}+B_{1c}K_{c}$ corresponding to $\lambda$. By construction,
$\sigma(A_{c}+B_{1c}K_{c})\cap\sigma(A_{u})=\emptyset$, we can solve
$q^{\top}_{u}$ from (\ref{eq2}) as $q^{\top}_{u}=p^{\top}_{c}(A_{cu}+B_{1c}K_{u})(\lambda I_{n_{u}}-A_{u})^{-1}$.
Substituting this into the left hand side of (\ref{eq3}) yields 
\[
p^{\top}_{c}[B_{2c}+(A_{cu}+B_{1c}K_{u})(\lambda I_{n_{u}}-A_{u})^{-1}B_{2u}].
\]
Denote 
\[
M(\lambda)=B_{2c}+A_{cu}(\lambda I-A_{u})^{-1}B_{2u}+B_{1c}K_{u}(\lambda I-A_{u})^{-1}B_{2u}.
\]
We show it is possible to construct $K_{u}$ such that $p^{\top}_{c}M(\lambda)\neq0$
for all $\lambda\in\sigma(A_{c}+B_{1c}K_{c})$, which in turn leads
to a contradiction in view of (\ref{eq3}).

Recall that the eigenvalues of $A_{c}+B_{1c}K_{c}$ are negative real,
which we now denote as $\{\lambda_{i}\}^{n_{c}}_{i=1}$, and let $\{p^{\top}_{i}\}^{n_{c}}_{i=1}$
be the corresponding left eigenvectors. The term $p^{\top}_{i}M(\lambda_{i})$
can be written as $v^{\top}_{i}+b^{\top}_{i}K_{u}F^{\top}_{i}$, in
which 
\begin{align*}
v^{\top}_{i} & =p^{\top}_{i}(B_{2c}+A_{cu}(\lambda_{i}I-A_{u})^{-1}B_{2u})\in\mathbb{R}^{1\times p},\\
b^{\top}_{i} & =p^{\top}_{i}B_{1c}\in\mathbb{R}^{1\times m},\\
F^{\top}_{i} & =(\lambda_{i}I-A_{u})^{-1}B_{2u}\in\mathbb{R}^{n_{u}\times p}.
\end{align*}
Define the sets,
\[
S_{i}=\{K_{u}\in\mathbb{R}^{m\times n_{u}}:F_{i}K^{\top}_{u}b_{i}+v_{i}=0\},
\]
and $S=S_{1}\cup\cdots\cup S_{n_{c}}$. Then each $S_{i}$ is nowhere
dense, neither is $S$, and as a result, the complement of $S$, i.e.,
the set of $K_{u}\in\mathbb{R}^{m\times n_{u}}$ for which $p^{\top}_{i}M(\lambda_{i})\ne0$
for all $i$, is dense.

In both cases, we have shown $[p^{\top}_{c},q^{\top}_{u}]=0$ by a
suitable choice of $K_{u}$, which makes the pair $(A+B_{1}K,B_{2})$
controllable. 
\end{proof}

\bibliographystyle{siam}
\bibliography{OT}

\end{document}